\documentclass[a4paper,10pt]{article}
\usepackage{amsmath,amsthm,xypic,
amssymb,amsfonts,
pdfsync,stmaryrd,mathrsfs,yfonts
}
\usepackage[T1]{fontenc}
\usepackage{tipa}
\newtheorem{thm}{Theorem}[section]

\newtheorem{prop}[thm]{Proposition}

\newtheorem{lem}[thm]{Lemma}
\theoremstyle{definition}

\newtheorem{ex}[thm]{Example}

\theoremstyle{plain}

\newcommand{\intt}{\mathop{\mathrm{Int}}}

\renewcommand{\hom}{\mathop{\mathrm{Hom}}}
\newcommand{\im}{\mathop{\mathrm{im}}}

\renewcommand{\phi}{\varphi}

\newcommand{\alt}{\mathop{\mathrm{Alt}}}

\newcommand{\symd}{\mathop{\mathrm{Symd}}}
\newcommand{\sym}{\mathop{\mathrm{Sym}}}

\newcommand{\End}{\mathop{\mathrm{End}}}

\newcommand{\car}{\mathop{\mathrm{char}}}
\newcommand{\id}{\mathop{\mathrm{id}}}

\author{A.-H. Nokhodkar}
\date{}
\title
{On the decomposition of metabolic involutions}

\begin{document}
\maketitle

\begin{abstract}
The problem of whether a metabolic idempotent of a central simple algebra with involution is contained in an invariant quaternion subalgebra is investigated.
As an application, the similar problem is studied for skew-symmetric elements whose squares lie in the square of the underlying field.
\\

\end{abstract}

\section{Introduction}
A hyperbolic involution on a central simple algebra is an involution which is adjoint to a hyperbolic hermitian form.
 It is known that hyperbolic involutions are decomposable (see \cite[(2.2)]{bayer} and \cite[Ch. II, Exercise 2]{knus}).
Indeed, using the ideas of \cite{bayer}, one can show that every hyperbolic idempotent of a central simple algebra with involution $(A,\sigma)$ lies in a $\sigma$-invariant quaternion subalgebra (see (\ref{hyp}) below).
A weaker condition for involutions is {\it metabolicity}.
Metabolic involutions were introduced in \cite{berhuy} and studied in more detail in \cite{dolphin}.
According to \cite[(4.8)]{dolphin}, metabolic involutions are adjoint to metabolic hermitian forms.

A decomposition problem concerning metabolic involutions is that whether a
meta\-bolic idempotent of a central simple algebra with involution $(A,\sigma)$ is contained in a $\sigma$-invariant quaternion subalgebra.
Since every metabolic idempotent $e$ satisfies $(e-\sigma(e))^2=1$,
such a quaternion subalgebra would also contain a skew-symmetric element whose square equals $1$.
Hence, a relevant decomposition problem arises as follows:
given a skew-symmetric element $u\in A$ satisfying $u^2=1$ (or more generally, $u^2\in F^2)$, is there any $\sigma$-invariant quaternion subalgebra of $A$ containing $u$?

The aim of this work is to study the aforementioned decomposition problems in arbitrary characteristic.
We start with some general observations on idempotents and square-central elements in a central simple algebra.
Let $A$ be a central simple algebra over a field $F$ and let $u\in A$ be an element satisfying $u^2=\lambda^2$ for some $\lambda\in F^\times$.
In \cite[(4.1)]{barry}, it was shown that if $\car F\neq2$, then $u$ is contained in a quaternion subalgebra of $A$ if and only if $\dim_F(u+\lambda)A=\dim_F(u-\lambda)A$.
Using similar methods we shall generalize this result in (\ref{qinv}) to arbitrary characteristic, includ\-ing $\lambda=0$.

In \S{\ref{sec-met}} we study metabolic idempotents of a central simple algebra $(A,\sigma)$.
The main result is (\ref{met}), which asserts that a metabolic idempotent $e\in A$ lies in a $\sigma$-invariant quaternion subalgebra if and only if it is hyperbolic or $\dim_Fe\sigma(e)A=\frac{1}{2}\dim_FA$.
Using this, we will see in (\ref{ex}) that in contrast with the case of hyperbolic idempotents, there exists a metabolic idempotent with respect to an involution $\sigma$ (over a field of arbitrary characteristic), which is not contained in any $\sigma$-invariant quaternion subalgebra.

In the last section we use the decomposition criterion (\ref{met}) to study skew-symmetric elements
of $A$ whose squares lie in $F^2$.
Consider an element $u\in A$ satisfying $u^2=\lambda^2$ for some $\lambda\in F$ and $\dim_F(u+\lambda)A=\frac{1}{2}\dim_FA$.
As we shall see in (\ref{umet}), if $u\in\alt(A,\sigma)$, $\sigma(\lambda)=\lambda$ and $\lambda\neq0$, then there exists a metabolic idempotent $e\in A$ such that $u=e-\sigma(e)$.
An important application of this result is (\ref{sym}) which asserts that if $\sigma(u)=-u$, then there exists a $\sigma$-invariant quaternion subalgebra of $A$ containing $u$, except for the case where $\car F=2$ and $\sigma$ is orthogonal.
Finally, this exceptional case is treated in (\ref{cor}) and (\ref{ex2}).

\section{Preliminaries}
Let $A$ be a central simple algebra over a field $F$ and let $\sigma$ be an {\it involution} on $A$, i.e., an anti-automorphism satisfying $\sigma^2=\id$.
We say that $\sigma$ is {\it of the first kind} if $\sigma|_F=\id$.
Otherwise, $\sigma$ is said to  be {\it of the second kind}.
Involutions of the second kind are also called {\it unitary} involutions.
After scalar extension to a splitting field of $A$, every involution $\sigma$ of the first kind becomes adjoint to a symmetric or skew-symmetric bilinear form.
If this bilinear form is alternating, we say that $\sigma$ is {\it symplectic}.
Otherwise, $\sigma$ is called {\it orthogonal}.
The set of all {\it symmetric} elements of an algebra with involution $(A,\sigma)$, i.e., elements fixed by $\sigma$, is denoted by $\sym(A,\sigma)$.
We also use the notations
\begin{align*}
  \symd(A,\sigma)=\{x+\sigma(x)\mid x\in A\} \quad{\rm and}\quad
  \alt(A,\sigma)=\{x-\sigma(x)\mid x\in A\}.
\end{align*}

An algebra with involution $(A,\sigma)$ (or simply the involution $\sigma$) is called {\it hyperbolic} if it has a {\it hyperbolic idempotent} with respect to $\sigma$, i.e., an idempotent $e\in A$ satisfying $\sigma(e)=1-e$.
Note that for every hyperbolic idempotent $e$ we have $\dim_FeA=\frac{1}{2}\dim_FA$ (see \cite[p. 75]{knus}).
An idempotent $e\in A$ is called {\it metabolic} if (1) $\sigma(e)e=0$ and
(2) $\dim_FeA=\frac{1}{2}\dim_FA$.
Either of these two conditions can be equivalently substituted with
$(1-e)(1-\sigma(e))=0$,
as observed in \cite{berhuy} and \cite{dolphin}.
The pair $(A,\sigma)$ (or simply the involution $\sigma$) is called {\it metabolic}
if there exists a  metabolic idempotent $e\in A$ with respect to $\sigma$.
Every hyperbolic idempotent is metabolic, but there exist metabolic idempotents (in all characteristics) that are not hyperbolic.
However, if $e$ is a metabolic idempotent and there exists an element $x\in A$ with $x+\sigma(x)=1$, then the element $e':=e-ex\sigma(e)$ is a hyperbolic idempotent of $(A,\sigma)$ (see \cite[(A.3)]{berhuy}).
Combining this observation with (\ref{pre}) below, one concludes that an involution $\sigma$ is hyperbolic if and only if it is metabolic, except for the case where $\car F=2$ and $\sigma$ is orthogonal.
Note that in view of \cite[(2.6 (2))]{knus}, hyperbolic involutions can never exist in this exceptional case.
\begin{lem}\label{pre}
Let $(A,\sigma)$ be a central simple algebra with involution over a field $F$.
Then there exists $x\in A$ such that $x+\sigma(x)=1$, except for the case where $\car F=2$ and $\sigma$ is orthogonal.
\end{lem}

\begin{proof}
If $\car F\neq2$, take $x=\frac{1}{2}$.
Otherwise, $\sigma$ is either symplectic or unitary and the existence of such an element follows from \cite[(2.6 (2)) and (2.17)]{knus}.
\end{proof}

\section{Idempotents and square-central elements}
Let $D$ be a central division algebra over a field $F$.
We write $M_n(D)$ for the algebra of all $n\times n$ matrices with entries in $D$.
The identity matrix in $M_n(D)$ is denoted by $I_n$.
Also, for $\alpha\in D$, we denote by $J_n(\alpha)$ the $n\times n$ {\it Jordan block} whose diagonal and superdiagonal entries equal to $\alpha$ and $1$ respectively and all other entries are zero.

Our first result is easily deduced from \cite[Prop. 6 and Thm. 7]{dok}.
\begin{lem}\label{jordan}
Let $V$ be a finite-dimensional right vector space over $D$ and let $x\in\End_D(V)$.
\begin{itemize}
  \item[(1)] If $x^2=\lambda^2$ for some $\lambda\in F$, then there exist integers $m,n,k$ and a basis of $V$ with respect to which the matrix of $x$ is a direct sum of $\lambda I_m$, $-\lambda I_n$ and $k$ Jordan blocks $J_2(\lambda)$ (the blocks $J_2(\lambda)$ appear only if $\car F=2$ or $\lambda=0$).
  \item[(2)] If $x^2=x$, then there exist an integer $m$ and a basis of $V$ with respect to which the matrix of $x$ is a direct sum of $I_m$ and a zero matrix.
\end{itemize}
\end{lem}

The next result generalizes \cite[(4.1)]{barry} to arbitrary characteristic.
Recall that a {\it quaternion algebra} over a field $F$ is a central simple $F$-algebra of degree $2$.

\begin{prop}{\rm (\cite[(4.1)]{barry})}\label{qinv}
Let $A$ be a central simple algebra over a field $F$ and let $u\in A\setminus F$ such that $u^2=\lambda^2\in F^2$ for some $\lambda\in F$.
There exists a split quaternion $F$-subalgebra of $A$ containing $u$ if and only if $\dim_F(\lambda+u)A=\frac{1}{2}\dim_FA$.
\end{prop}

\begin{proof}
If $\car F\neq2$ and $\lambda\neq0$, the result is easily deduced from \cite[(4.1)]{barry}, hence let $\car F=2$ or $\lambda=0$.
Suppose that $u$ is contained in a split quaternion subalgebra of $A$.
We identify this subalgebra with $\End_F(W)$ for some $2$-dimensional vector space $W$ over $F$.
As $u\notin F$, by (\ref{jordan} (1)) there exists a basis of $W$ with respect to which the matrix of $u$ equals the Jordan block $J_2(\lambda)$.
Since $\lambda=0$ or $\car F=2$, $\lambda+u$ is represented in this basis by $J_2(0)$.
Thus, $\dim_F(\lambda+u)\End_F(W)=2=\frac{1}{2}\dim_F\End_F(W)$, which leads to the desired equality.

To prove the converse, observe first that the conditions $u\notin F$ and $u^2\in F^2$ imply that $A$ is not a division algebra.
Choose a division $F$-algebra $D$, Brauer equivalent to $A$.
Then we may identify $A=\End_D(V)$, where $V$ is a right $D$-vector space.
Since $u^2=\lambda^2$, by (\ref{jordan} (1)) there exist an integer $r$ and a basis $\mathcal{B}$ of $V$ with respect to which the matrix of $u$ is a direct sum of $\lambda I_r$ and the Jordan blocks $J_2(\lambda)$.
Let $s$ be the number of blocks $J_2(\lambda)$ in this decomposition, so that $r+2s=\dim_DV$, where $\dim_DV$ denotes the right dimension of $V$ over $D$.
Since $\lambda=0$ or $\car F=2$, the matrix of $\lambda+u$ in $\mathcal{B}$ is a direct sum of $s$ blocks $J_2(0)$ and an $r\times r$ zero matrix, hence $\dim_D(\lambda+u)A=s(r+2s)$.
The equality $\dim_D(\lambda+u)A=\frac{1}{2}\dim_DA=\frac{1}{2}(r+2s)^2$ then implies that $r=0$.
Thus, $u$ is a direct sum of the Jordan blocks $J_2(\lambda)$
(in particular, $\dim_DV$ is even).
Let $v\in \End_D(V)$ be the endomorphism which is represented in $\mathcal{B}$ by direct sums of matrices of the form
\[\left(\begin{array}{cc}0 & 0 \\1 &0\end{array}\right).\]
Then the elements $1$, $u$, $v$ and $uv$ span a subalgebra of $A=\End_D(V)$, isomorphic to $M_2(F)$.
\end{proof}

\begin{lem}\label{w}
Let $A$ be a central simple algebra over a field $F$ and let $e\in A$ be an idempotent.
Suppose that there exists a nonzero element $u\in A$ satisfying $u^2=\lambda^2$ for some $\lambda\in F$, $ue=\lambda e$ and $e(\lambda+u)=\lambda+u$.
If $\dim_FeA=\dim_Fe(u-\lambda)A=\frac{1}{2}\dim_FA$, then there exists an element $w\in A$ such that
$w^2=0$, $ew=0$, $we=w$, $uw=e-\lambda w$ and $wu=\lambda w-e+1$.
Furthermore, the elements $1$, $e$, $u$ and $w$ span a split quaternion subalgebra of $A$.
\end{lem}

\begin{proof}
We may identify $A=\End_D(V)$, where $D$ is a division $F$-algebra Brauer equivalent to $A$ and $V$ is a right vector space over $D$.
Let $n=\dim_DV$.
The condition $\dim_FeA=\frac{1}{2}\dim_FA$ implies that $e$ is a nontrivial idempotent.
Hence $A$ is not a division algebra, i.e., $n>1$.
By (\ref{jordan}) there exist an integer $m$ and a $D$-basis $\mathcal{B}$ of $V$, with respect to which $e$ is represented by the matrix
\[e=\left(\begin{array}{cc} I_m & 0 \\0 & 0\end{array}\right)\in M_n(D).\]
Since  $\dim_FeA=\frac{1}{2}\dim_FA$, we have $m=\frac{n}{2}$.
Let
\[\left(\begin{array}{cc} U & W \\X & Y\end{array}\right),\]
denote the matrix of $u$ in $\mathcal{B}$, where $U,W,X,Y\in M_m(D)$.
The relations $ue=\lambda e$ and  $e(\lambda+u)=\lambda+u$ imply that $U=\lambda I_m$, $X=0$ and $Y=-\lambda I_m$, i.e.,
\[u=\left(\begin{array}{cc} \lambda I_m & W \\0 & -\lambda I_m\end{array}\right),\]
which yields
\[e(u-\lambda)=\left(\begin{array}{cc} 0 & W \\0 & 0\end{array}\right).\]
It follows that
\[e(u-\lambda)M_n(D)=\left\{\left(\begin{array}{cc} WM & WN\\ 0 & 0 \end{array}\right)\mid M,N\in M_m(D)\right\}.\]
The equality $\dim_Fe(u-\lambda)A=\frac{1}{2}\dim_FA$ implies that $W$ has a right inverse, hence it is invertible by \cite[\S19, Thm. 3]{draxl}.
A direct calculation then shows that the element
\[w=\left(\begin{array}{cc} 0 & 0 \\W^{-1} & 0\end{array}\right),\]
has the desired properties.
\end{proof}

\section{Metabolic and hyperbolic idempotents}\label{sec-met}
Let $(D,\bar{\ }\hspace{.05cm})$ be a division algebra with involution over a field $F$ and let $V$ be a right $D$-vector space.
Denote by $V^*=\hom_D(V,D)$, the dual space of $V$, with a right $D$-vector space structure $v^*d(v)=\bar dv^*(v)$ for $v\in V$, $v^*\in V^*$ and $d\in D$.
For every $\lambda\in F$ with $\lambda\bar\lambda=1$, define the $\lambda$-hermitian form $h_\lambda$ on $H_\lambda(V):=V^*\oplus V$ via
\[h_\lambda((v_1^*,v_1),(v_2^*,v_2))=v_1^*(v_2)+\lambda\overline{v_2^*(v_1)}.\]
A $\lambda$-hermitian space is called {\it hyperbolic} if it is isometric to $H_\lambda(V)$ for some right vector space $V$ over $D$.

The following result is implicitly contained in \cite[pp. 466-468]{bayer} for characteristic different from $2$.

\begin{thm}{\rm (\cite{bayer})}\label{hyp}
Let $(A,\sigma)$ be a central simple algebra with involution over a field $F$.
For every hyperbolic idempotent $e\in A$ there exists a $\sigma$-invariant split quaternion subalgebra of $A$ containing $e$.
\end{thm}

\begin{proof}
Let $D$ be a division algebra Brauer-equivalent to $A$ and let $\bar{} : D\rightarrow D$ be an involution on $D$ such that $\bar{}\ |_F=\sigma|_F$.
We may identify
\[(A,\sigma)=({\End}_V(D),\sigma_h),\]
where $V$ is a right vector space over $D$, $\lambda$ is an element in $D$ with $\lambda\bar\lambda=1$, $h:V\times V\rightarrow D$ is a non-degenerate $\lambda$-hermitian form and $\sigma_h$ is the adjoint involution of $\End_D(V)$ with respect to $h$.
Let $S=\im e\subseteq V$ and $T=\im \sigma(e)\subseteq V$ be the images of $e$ and $\sigma(e)$.
Define a map $\phi:S\rightarrow T^*$ via $\phi(s)(t)=h(s,t)$ for $s\in S$ and $t\in T$.
As observed in the proof of \cite[(2.1)]{bayer}, the map $\phi\oplus\id$ defines an isometry between $(V,h)=(S\oplus T,h)$ and $(H_\lambda(T),h_\lambda)$.
Considering this isomorphism as an identification, we have
\[e(t^*,t)=(t^*,0)\quad {\rm and}\quad\sigma(e)(t^*,t)=(0,t) \quad {\rm for}\ t^*\in T^*\ {\rm and} \ t\in T. \]
Let $\ell :T\times T\rightarrow D$ be a non-degenerate $1$-hermitian form and let $\theta:T\rightarrow T^*$ be the isomorphism of right $D$-vector spaces defined by $\theta(t_1)(t_2)=\ell(t_1,t_2)$ for $t_1,t_2\in T$.
Define $u,v\in A=\End_D(V)$ via
\[u(t^*,t)=(\theta(t),0)\quad {\rm and}\quad v(t^*,t)=(0,\theta^{-1}(t^*)).\]
As observed in the proof of \cite[(2.2)]{bayer}, the elements $e$, $\sigma(e)$, $u$ and $v$
span a $\sigma$-invariant subalgebra of $A$, isomorphic to $M_2(F)$.
\end{proof}

\begin{lem}\label{hyp1}
  Let $(A,\sigma)$ be a central simple algebra with involution over a field $F$ and let $e\in A$  be a hyperbolic idempotent.
  Suppose that there exists $u\in A$ satisfying $u^2=0$, $\dim_FuA=\frac{1}{2}\dim_FA$, $ue=0$ and $eu=u$.
  If $\sigma(u)=\pm u$, then there exists a $\sigma$-invariant quaternion subalgebra of $A$ containing $u$ and $e$.
\end{lem}

\begin{proof}
Since $e$  is a hyperbolic idempotent, we have $\dim_FeA=\frac{1}{2}\dim_FA$.
Hence, using (\ref{w}) with $\lambda=0$ one can find $w\in A$ such that
\[w^2=0,\quad ew=0,\quad we=w,\quad uw=e\quad {\rm and}\quad wu=1-e=\sigma(e).\]
Also, the $F$-algebra $Q\subseteq A$ spanned by $1$, $e$, $u$ and $w$ is a quaternion algebra.
We claim that $\sigma(Q)=Q$.
Since $\sigma(u)=\pm u\in Q$ and $\sigma(e)=1-e\in Q$, it suffices to show that $\sigma(w)\in Q$.
Keeping the notations of the proof of (\ref{w}) with $\lambda=0$ we have
\[u=\left(\begin{array}{cc} 0 & W \\0 & 0\end{array}\right)\quad {\rm and}\quad\sigma(e)=1-e=\left(\begin{array}{cc} 0 & 0 \\0 & I_m\end{array}\right),\]
where $W\in M_m(D)$ is invertible.
Write
\[\sigma(w)=\left(\begin{array}{cc} X & Y \\Z & T\end{array}\right),\]
for some $X,Y,Z,T\in M_m(D)$.
Suppose first that $\sigma(u)=u$.
Applying $\sigma$ on $uw=e$, we get $\sigma(w)u=\sigma(e)$, which implies that $X=0$ and $Z=W^{-1}$.
Similarly, applying $\sigma$ on the equality $wu=\sigma(e)$, we obtain $u\sigma(w)=e$, hence $T=0$.
Finally, as $w^2=0$, we have $\sigma(w)^2=0$, which leads to $Y=0$.
It follows that
\[\sigma(w)=\left(\begin{array}{cc} 0 & 0 \\W^{-1} & 0\end{array}\right)=w\in Q.\]
If $\sigma(u)=-u$, a similar argument shows that $\sigma(w)=-w\in Q$.
\end{proof}

\begin{lem}\label{prop}
  Let $(A,\sigma)$ be a central simple algebra with involution over a field $F$ and let $e\in A$ be a metabolic idempotent.
  \begin{itemize}
    \item[(1)] $e$ is a hyperbolic idempotent if and only if $e\sigma(e)=0$.
    \item[(2)] For every $x\in A$, the element $e':=e-ex\sigma(e)$ is a metabolic idempotent.
  \end{itemize}
\end{lem}
\begin{proof}
  The first statement follows from the equality $1-e-\sigma(e)+e\sigma(e)=0$.
  To prove (2) note that the equality $\sigma(e)e=0$ implies that $e'$ is an idempotent and
$\sigma(e')e'=0$.
Also, since $ee'=e'$ and $e'e=e$, we have $eA=e'A$.
It follows that $\dim_Fe'A=\dim_FeA=\frac{1}{2}\dim_FA$.
Hence $e'$ is a metabolic idempotent.
\end{proof}

\begin{thm}\label{met}
  Let $(A,\sigma)$ be a central simple algebra with involution over a field $F$ and let $e\in A$ be a metabolic idempotent.
  There exists a $\sigma$-invariant split quaternion subalgebra of $A$ containing $e$ if and only if either $e$ is hyperbolic or $\dim_Fe\sigma(e)A=\frac{1}{2}\dim_FA$.
\end{thm}

\begin{proof}
Set $u=e\sigma(e)$.
Suppose that there exists a $\sigma$-invariant quaternion subalgebra of $A$ containing $e$.
If $e$ is not hyperbolic, then we have $u\neq0$ by (\ref{prop} (1)).
As $u^2=0$, using (\ref{qinv}) with $\lambda=0$ we obtain $\dim_FuA=\frac{1}{2}\dim_FA$.
Conversely, if $e$ is hyperbolic the result follows from (\ref{hyp}).
Suppose that $\dim_FuA=\frac{1}{2}\dim_FA$.
We proceed similar to the proof of (\ref{hyp1}).
We have $u^2=0$, $ue=0$, $eu=u$ and $\dim_FeuA=\dim_FuA=\frac{1}{2}\dim_FA$.
Also, $\dim_FeA=\frac{1}{2}\dim_FA$, since $e$ is a metabolic idempotent.
Thus, one can use (\ref{w}) with $\lambda=0$ to find $w\in A$ such that
\[w^2=0,\quad ew=0,\quad we=w,\quad uw=e\quad {\rm and}\quad wu=1-e.\]
By (\ref{w}), the $F$-algebra $Q\subseteq A$ spanned by $1$, $e$, $u$ and $w$ is a quaternion algebra.
We claim that $\sigma(Q)=Q$.
We have $\sigma(u)=u\in Q$.
Also, the equality $(1-e)(1-\sigma(e))=0$ implies that $\sigma(e)=1-e+u\in Q$.
Thus, it remains to show that $\sigma(w)\in Q$.
Keep the notation of the proof of (\ref{w}) with $\lambda=0$, so that
\[u=\left(\begin{array}{cc} 0 & W \\0 & 0\end{array}\right)\quad {\rm and}\quad\sigma(e)=1-e+u=\left(\begin{array}{cc} 0 & W \\0 & I_m\end{array}\right),\]
where $W$ is an invertible matrix in $M_m(D)$.
Write
\[\sigma(w)=\left(\begin{array}{cc} X & Y \\Z & T\end{array}\right),\]
for some $X,Y,Z,T\in M_m(D)$.
Applying $\sigma$ on $uw=e$, we get $\sigma(w)u=\sigma(e)$, which leads to $X=I_m$ and $Z=W^{-1}$.
Also, applying $\sigma$ on $wu=1-e$, we obtain $u\sigma(w)=1-\sigma(e)$, so $T=-I_m$.
Finally, the equality $w^2=0$ implies that $\sigma(w)^2=0$, hence $Y=-W$.
It follows that
\[\sigma(w)=\left(\begin{array}{cc} I_m & -W \\W^{-1} & -I_m\end{array}\right)=w+2e-u-1\in Q.\qedhere\]
\end{proof}

The next example shows the existence of a metabolic idempotent in a central simple algebra with involution over a field of arbitrary characteristic which is not contained in any invariant quaternion subalgebra.
\begin{ex}\label{ex}
Let $F$ be a field and let $t$ be the transpose involution on $M_4(F)$.
Set $\sigma=\intt(u)\circ t$, where
\begin{equation*}
 u=\left(\begin{array}{cccc} 1 & 0 & 0 & 0 \\0 & -1 & 0 & 0\\ 0 & 0 & 1 & 0 \\0 & 0 & 0 & -1\end{array}\right),
\end{equation*}
and $\intt(u)$ is the inner automorphism of $M_4(F)$ induced by $u$, i.e., $\intt(u)(x)=uxu^{-1}$ for $x\in M_4(F)$.
Since $u$ is a symmetric matrix, by \cite[(2.7 (1))]{knus}, $\sigma$ is an involution of the first kind on $M_4(F)$.
A straightforward calculation shows that the element
\begin{equation*}
 e:=\left(\begin{array}{cccc} 1 & 0 & 0 & 0 \\1 & 0 & 0 & 0\\ 1 & -1 & 1 & 0 \\1 & -1 & 1 & 0\end{array}\right),
\end{equation*}
is a metabolic idempotent with respect to $\sigma$.
We also have
 \begin{equation*}
 e\sigma(e)=\left(\begin{array}{cccc} 1 & -1 & 1 & -1 \\1 & -1 & 1 & -1\\1 & -1 & 1 & -1 \\1 & -1 & 1 & -1\end{array}\right),
\end{equation*}
hence $e$ is not hyperbolic and $\dim_Fe\sigma(e)M_4(F)=4\neq\frac{1}{2}\dim_FM_4(F)$.
By (\ref{met}) there is no $\sigma$-invariant quaternion subalgebra of $M_4(F)$ containing $e$.
\end{ex}

\section{Applications to square-central elements}
Throughout this section, $(A,\sigma)$ denotes a central simple algebra with involution over a field $F$ and $u\in A\setminus F$ is an element satisfying
$u^2=\lambda^2$ for some $\lambda\in F$ with $\sigma(\lambda)=\lambda$ and
\begin{equation}\label{eq5}
\dim_F(\lambda+u)A=\frac{1}{2}\dim_FA.
\end{equation}
Our purpose is to find some sufficient conditions on $u$ to be contained in a $\sigma$-invariant quaternion subalgebra of $A$.
Note that in view of (\ref{qinv}), the condition (\ref{eq5}) is necessary for this aim.
Also, if $\sigma$ is of the first kind, then the equality $\sigma(\lambda)=\lambda$ holds automatically.

For a right ideal $I$ of $A$ let $I^\perp=\{x\in A\mid\sigma(x)y=0 \ {\rm for}\ y\in I\}$.
Then $I^\perp$ is also a right ideal of $A$.
Moreover, according to \cite[(6.2)]{knus} we have
\[\dim_FI+\dim_FI^\perp=\dim_FA.\]
\begin{prop}\label{umet}
Suppose that $\sigma(u)=-u$. Then
\begin{itemize}
  \item[(1)] $\sigma$ is metabolic.
  \item[(2)] There exists a metabolic idempotent $e\in A$ such that $e(\lambda+u)=\lambda+u$ and $ue=\lambda e$.
  \item[(3)] If $\lambda\neq0$ and $u\in\alt(A,\sigma)$, then there exists a metabolic idempotent $e'\in A$ such that $\lambda^{-1}u=e'-\sigma(e')$.
      In addition, if $\car F=2$, such an idempotent satisfies $\dim_Fe'\sigma(e')=\frac{1}{2}\dim_FA$.
      If $\car F\neq2$, the idempotent $e'\in A$ can be chosen to be hyperbolic.
\end{itemize}
\end{prop}

\begin{proof}
  Consider the right ideal $I:=(\lambda+u)A$ of $A$.
  Since $\sigma(\lambda+u)=\lambda-u$, we have $\sigma(I)I=0$.
  By dimension count we obtain $I=I^\perp$, hence $\sigma$ is metabolic by \cite[(4.4)]{dolphin}.
  This proves (1).

(2) By \cite[(1.13)]{knus} there exists an idempotent $e\in A$ such that $I=eA$.
As $I=I^\perp$, we have $\sigma(e)e=0$.
We also have $\dim_FeA=\dim_F(\lambda+u)A=\frac{1}{2}\dim_FA$, hence $e$ is a metabolic idempotent.
Since $(\lambda+u)A=eA$, there exist $x,y\in A$ such that $\lambda+u=ex$ and $e=(\lambda+u)y$.
Multiplying the first equality on the left by $e$, we get $e(\lambda+u)=e^2x=ex=\lambda+u$.
Similarly, multiplying the second equality on the left by $u$, we obtain $ue=\lambda e$.

(3) Applying $\sigma$ on $e(\lambda+u)=\lambda+u$ we get $(\lambda-u)\sigma(e)=\lambda-u$, i.e., $u\sigma(e)=\lambda\sigma(e)+u-\lambda$.
Also, $eu=e(\lambda+u)-\lambda e=\lambda+u-\lambda e$, hence
\begin{align}\label{eq1}
eu\sigma(e)&=e(\lambda\sigma(e)+u-\lambda)=\lambda e\sigma(e)+e u-\lambda e=\lambda e\sigma(e)+\lambda+u-2\lambda e.
\end{align}
Now, write $\lambda^{-1}u=x-\sigma(x)$ for some $x\in A$ and set $e'=e-e\sigma(x)\sigma(e)$.
By (\ref{prop} (2)), $e'$ is a metabolic idempotent.
Using the relation (\ref{eq1}) we get
\begin{align*}
  e'-\sigma(e')&=e-e\sigma(x)\sigma(e)-\sigma(e)+ex\sigma(e)=e-\sigma(e)+e(x-\sigma(x))\sigma(e)\\
  &=e-\sigma(e)+\lambda^{-1}eu\sigma(e)=e-\sigma(e)+e\sigma(e)+1+\lambda^{-1}u-2e\\
&=1-e-\sigma(e)+e\sigma(e)+\lambda^{-1}u=(1-e)(1-\sigma(e))+\lambda^{-1}u=\lambda^{-1}u.
\end{align*}
If $\car F=2$, the above equalities together with $(1-e')(1-\sigma(e'))=0$ imply that $e'\sigma(e')=1+\lambda^{-1}u$.
Thus,
\[\dim_Fe'\sigma(e')A=\dim_F(1+\lambda^{-1}u)A=\dim_F(\lambda+u)A=\frac{1}{2}\dim_FA.\]
If $\car F\neq2$, the element $h:=e'-\frac{1}{2}e'\sigma(e')$ is a hyperbolic idempotent with respect to $\sigma$.
We also have $h-\sigma(h)=e'-\sigma(e')=\lambda^{-1}u$, completing the proof.
\end{proof}

\begin{prop}\label{sq}
If $u\in\alt(A,\sigma)$, then there exists a $\sigma$-invariant split quaternion $F$-subalgebra of $A$ containing $u$,
except for the case where $\car F=2$, $\lambda=0$ and $\sigma$ is orthogonal.
In this case there is no $\sigma$-invariant quaternion subalgebra $Q\subseteq A$ with $u\in Q$.
\end{prop}
\begin{proof}
Suppose first that  $\lambda\neq0$.
If $\car F\neq2$, by (\ref{umet} (3)) there exists a hyperbolic idempotent $h\in A$ such that $\lambda^{-1}u=h-\sigma(h)$ and the result follows from (\ref{hyp}).
If $\car F=2$, by (\ref{umet} (3)) there exists a metabolic idempotent $e\in A$ such that $\lambda^{-1}u=e-\sigma(e)$ and
$\dim_Fe\sigma(e)A=\frac{1}{2}\dim_FA$.
Hence the conclusion follows from (\ref{met}).

Now, let $\lambda=0$.
The hypotheses imply that either $\car F\neq2$ or $\sigma$ is not orthogonal.
In either case, by (\ref{pre}) there exists $x\in A$ such that $x+\sigma(x)=1$.
Using (\ref{umet} (2)) with $\lambda=0$, one can find a metabolic idempotent $e\in A$ such that $eu=u$ and $ue=0$.
Then $h:=e-ex\sigma(e)$ is a hyperbolic idempotent with respect to $\sigma$.
As $ue=0$ we have $uh=0$.
Also, the relations $\sigma(u)=-u$ and $ue=0$ imply that $\sigma(e)u=0$, hence $hu=eu=u$.
Hence the result follows from (\ref{hyp1}).

Finally, consider the exceptional case, i.e., suppose that $\car F=2$, $\lambda=0$ and $\sigma$ is orthogonal.
If there exists a $\sigma$-invariant quaternion subalgebra $Q$ of $A$ containing $u$, then $u\in\alt(Q,\sigma|_Q)$ by \cite[(3.5)]{mn1}.
By dimension count we obtain $\alt(Q,\sigma|_Q)=Fu$.
But this implies that $\alt(Q,\sigma|_Q)$ does not contain any invertible element, contradicting \cite[(2.8 (2))]{knus}.
\end{proof}

\begin{thm}\label{sym}
If $\sigma(u)=-u$, then there exists a $\sigma$-invariant split quaternion $F$-subalgebra of $A$ containing $u$, except for the case where $\car F=2$ and $\sigma$ is orthogonal.
\end{thm}

\begin{proof}
  As observed in \cite[p.  14]{knus} and \cite[(2.17)]{knus}, if $\car F\neq2$ or $\sigma$ is unitary, the equality $\sigma(u)=-u$ implies that $u\in\alt(A,\sigma)$ and the result follows from (\ref{sq}).
  Suppose that $\car F=2$ and $\sigma$ is symplectic.
  Replacing $u$ with $u+1$ if necessary, we may assume that $\lambda\neq0$ (note that since $\car F=2$, this replacement does not change the assumptions $\sigma(u)=-u$ and $\dim_F(\lambda+u)A=\frac{1}{2}\dim_FA$).
  If $u\in\alt(A,\sigma)$, the conclusion follows from (\ref{sq}),
  hence suppose that $u\notin\alt(A,\sigma)$.
  Set $\tau=\intt(u)\circ\sigma$.
  By \cite[(2.7 (3))]{knus}, $\tau$ is an orthogonal involution on $A$.
  Since $\sigma$ is symplectic, by \cite[(2.6 (2))]{knus} we have $1\in\alt(A,\sigma)$.
  Hence \cite[(2.7 (2))]{knus} implies that $u\in\alt(A,\tau)$.
  By (\ref{sq}), there exists a $\tau$-invariant quaternion subalgebra $Q\subseteq A$ containing $u$.
  As $u\in Q$ and $\tau(Q)=Q$, we get $\sigma(Q)=u^{-1}\tau(Q)u=Q$, completing the proof.
\end{proof}

\begin{lem}\label{quat}
  Let $\car F=2$ and let $(Q,\tau)$ be a quaternion algebra with ortho\-gonal involution over $F$.
  For every symmetric element $x\in Q$ with $x^2\in F$, there exists $\alpha\in F$ such that $x+\alpha\in\alt(Q,\tau)$.
\end{lem}
\begin{proof}
By \cite[(6.1)]{mn}, $Q$ is generated as an $F$-algebra by two elements $w,v\in Q$ satisfying
\[w^2\in F^\times,\quad v^2\in F,\quad wv+vw=1,\quad {\rm and}\quad \tau(v)+v=w^{-1}.\]
Hence, $\alt(Q,\tau)=Fw$ and $\sym(Q,\tau)=F+Fw+Fwv$.
Since $x\in\sym(Q,\tau)$ and $x^2\in F$, we have $x=\alpha+\beta w$ for some $\alpha,\beta\in F$.
Thus, $x+\alpha=\beta w\in\alt(Q,\tau)$.
\end{proof}
We now consider the exceptional case in (\ref{sym}).
\begin{prop}\label{cor}
  Suppose that $\car F=2$, $\sigma$ is orthogonal and $\sigma(u)=u$.
  There exists a $\sigma$-invariant quaternion subalgebra of $A$ containing $u$ if and only if $u+\alpha\in\alt(A,\sigma)$ for some $\alpha\in F$ with $\alpha\neq\lambda$.
\end{prop}

\begin{proof}
  Suppose that $u\in Q$ for some $\sigma$-invariant quaternion subalgebra $Q$ of $A$.
  As $\sigma(u)=u$, by (\ref{quat}) there exists $\alpha\in F$ such that $u+\alpha\in\alt(Q,\sigma|_Q)\subseteq\alt(A,\sigma)$.
  Since $(u+\alpha)^2=\lambda^2+\alpha^2$, using (\ref{sq}) we obtain $\alpha\neq\lambda$.
  To prove the converse, let $u'=u+\alpha\in\alt(A,\sigma)$ and $\lambda'=\alpha+\lambda$.
  Then $u^{\prime2}=\lambda^{\prime2}\in F^{\times2}$ and $\dim_F(u'+\lambda')A=\dim_F(u+\lambda)A=\frac{1}{2}\dim_FA$.
  The conclusion therefore follows from (\ref{sq}).
\end{proof}

We conclude with an example of a symmetric element $u$ in a central simple algebra with involution $(A,\sigma)$ satisfying $u^2=\lambda^2\neq0$ and $\dim_F(u+\lambda)A=\frac{1}{2}\dim_FA$, which is not contained in any $\sigma$-invariant quaternion subalgebra of~$A$.

\begin{ex}\label{ex2}
  Suppose that $\car F=2$, $\lambda\neq0$ and $(A,\sigma)=(M_4(F),t)$.
  Let
  \[u=\left(\begin{array}{cccc} \lambda & 0 & \lambda & \lambda \\0 & \lambda & \lambda & \lambda\\ \lambda & \lambda & 0 & \lambda \\\lambda & \lambda & \lambda & 0\end{array}\right).\]
  Then $u$ is a symmetric matrix, $u^2=\lambda^2\in F^2$ and $\dim_F(u+\lambda)A=\frac{1}{2}\dim_FA$.
  Since $\alt(A,\sigma)$ consists of all symmetric matrices with zero diagonal, we get $u+\alpha\notin \alt(A,\sigma)$ for every $\alpha\in F$.
  Thus, by (\ref{cor}) there is no $\sigma$-invariant quaternion subalgebra of $A$ containing $u$.
\end{ex}

\noindent A.-H. Nokhodkar, {\tt
    a.nokhodkar@kashanu.ac.ir}\\
Department of Pure Mathematics, Faculty of Science, University of Kashan, P.~O. Box 87317-53153, Kashan, Iran.


\small
\begin{thebibliography}{MM}
\bibitem{barry} D. Barry, Power-central elements in tensor products of symbol algebras.
LAG preprint server, http://www.math.uni-bielefeld.de/LAG/man /516.html.

\bibitem{bayer} E. Bayer-Fluckiger, D. B. Shapiro, J.-P. Tignol, Hyperbolic involutions. {\it Math. Z.} {\bf 214} (1993), no. 3, 461--476.

\bibitem{berhuy}
G. Berhuy,  C. Frings,  J.-P. Tignol,
Galois cohomology of the classical groups over imperfect fields.
{\it J. Pure Appl. Algebra} {\bf 211} (2007), no. 2, 307--341.

\bibitem{dok}
\DH okovi\'c, Dragomir \v Z. Inner derivations of division rings and canonical Jordan form of triangular operators. {\it Proc. Amer. Math. Soc.} {\bf 94} (1985), no. 3, 383--386.

\bibitem{dolphin} A. Dolphin, Metabolic involutions. {\it J. Algebra} {\bf 336} (2011), 286--300.

\bibitem{draxl} P. K. Draxl, Skew fields. London Mathematical Society Lecture Note Series, 81. {\it Cambridge University Press, Cambridge}, 1983.


\bibitem{knus} M.-A. Knus, A. S. Merkurjev, M. Rost, J.-P. Tignol, {\it The book of involutions}. American Mathematical Society Colloquium Publications, 44. American Mathematical Society, Providence, RI, 1998.

\bibitem{mn} M. G. Mahmoudi, A.-H. Nokhodkar, Involutions of a Clifford algebra induced by involutions of orthogonal group in characteristic $2$. {\it Comm. Algebra} {\bf 43} (2015), no. 9, 3898--3919.

\bibitem{mn1} M. G. Mahmoudi, A.-H. Nokhodkar, On split products of quaternion algebras with involution in characteristic two.
{\it J. Pure Appl. Algebra} {\bf 218} (2014), no. 4, 731--734.
\end{thebibliography}
\end{document}